\documentclass[microtype]{amsart}

\newcommand{\co}{\!:}

\usepackage[shortlabels]{enumitem}  
\setlist[enumerate]{nolistsep, topsep=-3pt}
\setlist[itemize]{nolistsep, topsep=-3pt}

\usepackage{amsmath, amssymb,euscript}
\usepackage{fouridx}
\renewcommand{\star}[1]{\ensuremath{{\fourIdx{*}{}{}{}{#1}}}}

\usepackage{hyperref}

\title{A density version of a theorem of Banach}

\author{David A. Ross}
\address{Department of Mathematics, University of Hawaii, Honolulu, HI 96822}
\email{ross@math.hawaii.edu}
\urladdr{https://math.hawaii.edu/~ross/}

\newcommand{\SA}{\ensuremath{\EuScript A}}
\newcommand{\SAL}{\ensuremath{\SA_L}}
\newcommand{\SAS}{\ensuremath{\SA_S}}
\newcommand{\st}[1]{\ensuremath{^\circ\negmedspace{#1}}}
\newcommand{\N}{\ensuremath{\mathbb N}}
\newcommand{\compl}[1]{\ensuremath{#1^\complement}}

\newtheorem{Thm}{Theorem}[section]
\newtheorem{Lemma}{Lemma}
\newtheorem{Cor}{Corollary}

\begin{document}

\begin{abstract}
The S--measure construction from nonstandard analysis is used to prove an extension of a result
on the intersection of sets in a finitely-additive measure space.  This is then used to give a density-limit version
of a representation theorem of Banach.
\end{abstract}



\maketitle

\section{Introduction}

The starting point for this note is the following result of Banach (described in Diestel and  Swart~\cite{diestelswart} as ``marvelous"):

\begin{Thm}\label{thm1}Let $X$ be a set, $B(X)$ be all bounded real functions on $X$, and $\{\,f_n : n\in\mathbb{N}\}$ be a uniformly
bounded sequence.  The following are equivalent:  (i)~$\{\,f_n\}_n$ converges weakly to $0$; (ii)~for any sequence
$\{x_k : k\in\mathbb{N}\}$ in $X$, $\lim\limits_{n\to\infty}\liminf\limits_{k\to\infty}f_n(x_k)=0$.
\end{Thm}

Weak convergence to zero here means that for any positive linear functional $T$ on $B(X)$, $Tf_n\to 0$ as $n\to\infty$.

The contrapositive is interesting:

\begin{Cor}\label{thm3}Let $X$ be a set, $B(X)$ be all bounded real functions on $X$, and $\{\,f_n : n\in\mathbb{N}\}$ be a uniformly
bounded sequence.  The following are equivalent:  (i)~there is a positive linear functional $T$ on $B(X)$, an infinite set $I\subseteq\N$, and an $r>0$ such that $|T(f_n)|>r$ for every $n\in I$; (ii)~there exists a sequence
$\{x_k : k\in\mathbb{N}\}$ in $X$, an infinite set $I\subseteq\N$, and an $r>0$ such that $\liminf\limits_{k\to\infty}|f_n(x_k)|>r$ for every $n\in I$. \end{Cor}

It is natural to consider the question of whether set $I$ can be required to have more structure than just being infinite; such a requirement would give a variant of weak convergence.  In this paper we adapt the proof from Ross~\cite{rossbanach} to prove a version of the Banach theorem in which $I$ is required to have positive upper density.

For $I\subseteq\N$ let $\bar{d}(I)=\limsup_n\|I\cap\{1,2,\dots,n\}\|/n$ (the \emph{upper asymptotic density} of $I$).  The main result of this paper is the following.

\begin{Thm}\label{thm4}Let $X$ be a set, $B(X)$ be all bounded real functions on $X$, and $\{\,f_n : n\in\mathbb{N}\}$ be a uniformly
bounded sequence.  The following are equivalent:  (i)~there is a positive linear functional $T$ on $B(X)$, an infinite set $I\subseteq\N$ with $\bar{d}(I)>0$, and an $r>0$ such that $|T(f_n)|>r$ for every $n\in I$; (ii)~there exists a sequence
$\{x_k : k\in\mathbb{N}\}$ in $X$, a set $I\subseteq\N$ with $\bar{d}(I)>0$, and an $r>0$ such that $\liminf\limits_{k\to\infty}|f_n(x_k)|>r$ for every $n\in I$. \end{Thm}

With appropriate choice of notation, this result can be made to look more like the original Banach result.  Write $d\lim_na_n=L$ provided it is \emph{not} the case that there is an $r>0$ and set $I\subseteq\N$ with $\bar{d}(I)>0$ such that $|f(n)-L|>r$ for all $n\in I$.  (For equivalent ways of writing such limits see Furstenberg \cite[Chapter 9]{Furstenberg}.)

Say that a sequence $f_n$ of functions in $X$ \emph{weakly d--converges to zero} provided that for any positive linear functional $T$ on $B(X)$, $d\lim_{n\to\infty}Tf_n = 0$.

Then result Theorem~\ref{thm4} becomes:

\begin{Cor}\label{thm2}Let $X$ be a set, $B(X)$ be all bounded real functions on $X$, and $\{\,f_n : n\in\mathbb{N}\}$ be a uniformly
bounded sequence.  The following are equivalent:  (i)~$\{\,f_n\}_n$ weakly d--converges to $0$; (ii)~for any sequence
$\{x_k : k\in\mathbb{N}\}$ in $X$, $d\!\lim\limits_{n\to\infty}\liminf\limits_{k\to\infty}f_n(x_k)=0$.
\end{Cor}

We note that there are many classical results for which replacing ``limit" by ``d--limit" yields an immediate open question.

Our proof uses nonstandard analysis, notably Abraham Robinson's notion of \emph{S--measurability}, and relies on a generalization (Corollary~\ref{bergelson}) to finitely-additive measures of a lemma from Bergelson~\cite{bergelson}.  The main idea is to substitute Corollary~\ref{bergelson} for the weaker \cite[Theorem 1]{rossbanach} in the proof of Theorem~\ref{thm1} in
Ross \cite{rossbanach}.
Nonstandard analysis has proved itself increasingly useful for the study of densities; see, for example, Jin \cite{Jin}.

\section{Loeb measures and S--measures}

The reader is assumed to be familiar with nonstandard analysis in general, and the Loeb measure construction in particular,
for example as in Ross \cite{Ross96}.  Assume that we work in a nonstandard model in the sense of Robinson, and that this model is as saturated
as it needs to be to carry out all constructions; in particular, it is an enlargement.

If $(X, \SA, \mu)$ is a finite measure space then both ${^*\SA}$ and $\SA_0=\{{\star{\!A}}:A\in\SA\}$ are algebras
on $\star{\!X}$.  Let $\SAS$ be the smallest $\sigma$--algebra\ containing ${\SA_0}$ and $\SAL$ be
the smallest $\sigma$--algebra\ containing ${^*\SA}$.  $(\star{\!X}, {^*\SA}, {^\circ{^*\mu}})$ is an external, standard, finitely-additive finite measure space, with ${^\circ{^*\mu}}({\star{\!X}})=\mu(X)<\infty$.  By either an appeal to the Carath\'eodory Extension Theorem or an elementary direct construction, $^{\circ*}{\mu}$ can be extended to a countably-additive measure (the Loeb measure) $\mu_L$ on $(\star{\!X}, \SAL)$, and by restriction on $(\star{\!X}, \SAS)$.

The algebra $\SAS$ of \emph{S--measurable} sets was introduced by Robinson \cite{Robinson}, then studied later by Henson and Wattenburg \cite{HW} (who used S--measurability to understand Egoroff's Theorem), and more recently by the author \cite{Ross96,R1,R2,Ross16}.

The main result we need is the following:

\begin{Lemma}(Henson and Wattenburg, 1981) \label{smeas}$\forall A\in\SA_S,\ $
\begin{align*}
\mu_L(A)&=\inf\{\mu(B)\ :\ A\subseteq{^*B}, B\in\SA\}\\
&=\sup\{\mu(B)\ :\ {^*B}\subseteq A,\ B\in{\SA}\}\\
&=\mu(A\cap X).
\end{align*}

\end{Lemma}

In particular, if $A\in\SAS$ and $A$ contains all standard points of $X$, then $\mu_L(A)=\mu(X)$ and $A$ contains sets of the form ${^*B}$ for $B\in\SA$ of arbitrary large measure.

\section{A fundamental lemma}

This section gives a new proof of a modest generalization (Corollary~\ref{bergelson}) of a lemma of Bergelson~\cite{bergelson} .  Bergelson's result is usually proved using Fatou's Lemma or the Lebesgue Dominated Convergence Theorem; the proof here replaces these with an appeal to Lemma~\ref{smeas}.  While this proof is not shorter than the standard ones, it is more explicit, which could prove useful in extending results which use it, such as Furstenburg's Multiple Ergodic Theorem \cite{Furstenberg}.

We begin with a weak form of the lemma.

\begin{Lemma}\label{weakbergelson}Let $(X, \SA, \mu)$ be a probability measure, $a>0$, $I_0\subseteq\N$ with $\bar{d}(I_0)>b>0$, and $A_n\in\SA$ with $\mu(A_n)\ge a$ for all $n\in I_0$. For some $I\subseteq I_0$ with $\bar{d}(I)\ge ab, \{A_n\}_{n\in I}$ has the finite intersection property.\end{Lemma}

Before proceeding with the proof, we note two immediate corollaries.  The first is Bergelson's original result, the second is the extension we need.

\begin{Cor}\label{bergelson0} Let $(X, \SA, \mu)$ be a probability measure, $a>0$, $I_0\subseteq\N$ with $\bar{d}(I_0)>b>0$ and $A_n\in\SA$ with $\mu(A_n)\ge a$ for all $n\in I_0$. For some $I\subseteq I_0$ with $\bar{d}(I)\ge ab$ and every finite $J\subseteq I$, $\mu(\hspace{.1em}\bigcap_{n\in J}A_n)>0$.
\end{Cor}

\begin{proof}Let $A_n'=A_n\setminus B$, where
\[B=\bigcup\left\{\bigcap_{i\in J}A_i : J\subseteq\mathbb{N}, J\textrm{ finite, } \mu\left(\hspace{.1em}\bigcap_{i\in J}A_i\right)=0 \right\}.\]
$B$ is a countable union of nullsets, so is itself a nullset.  $\mu(A_n')=\mu(A_n)\ge a$, and for any finite $J$,
$\mu(\hspace{.1em}\bigcap_{n\in J}A_n)>0$ if and only if $\bigcap_{n\in J}A_n'\neq\emptyset$.  Apply Lemma~\ref{weakbergelson} to the sequence $\{A_n'\}_n$ to get an index set $I$ with density at least $ab$ such that $\{A_n'\}_{n\in I}$ has the finite intersection property, then every finite intersection from $\{A_n\}_{n\in I}$ has positive measure.\end{proof}

\begin{Cor}\label{bergelson}Let $(X, \SA, \mu)$ be a \emph{finitely additive} probability measure, $a>0$, $I_0\subseteq\N$ with $\bar{d}(I_0)>b>0$, and $A_n\in\SA$ with $\mu(A_n)\ge a$ for all $n\in I_0$.  For some $I\subseteq I_0$ with $\bar{d}(I)\ge ab$ and every finite $J\subseteq I$, $\mu(\hspace{.1em}\bigcap_{n\in J}A_n)>0$.
\end{Cor}

\begin{proof}Let $({\star{\!X}},\SA_L,\mu_L)$ be the Loeb measure constructed from $(X, \SA, \mu)$.  Apply Corollary~\ref{bergelson0} to the sequence $\{{\star{\!A}_n}\}_{n\in\N}$ to get an index set $I\subseteq\N$ with $\bar{d}(I)\ge ab$ such that for every finite $J\subseteq I$, $\mu(\hspace{.1em}\bigcap_{n\in J}{\star{\!A}}_n)>0$.
The observation that $\mu(\hspace{.1em}\bigcap_{n\in J}A_n)=\mu_L(\hspace{.1em}\bigcap_{n\in J}{\star{\!A}}_n)$ for any finite $J\subseteq I$ completes the proof.\end{proof}

\subsection{Proof of Lemma~\ref{weakbergelson}}

Adopt the following notation.  If $J\subseteq\N$ and $N\in\N$  write $J\wedge N=J\cap\{1,\dots,N\}$ and $|J\wedge N|=$ the (finite) cardinality of $J\wedge N$.

Given $I_0$ with $\bar{d}(I_0)>b$ let $\eta_n$ be an increasing sequence of natural numbers with $|I_0\wedge \eta_n|/\eta_n>b$ for all $n$.  For $x\in X$ and $n\in\N$ define:
\[F_n=\frac{1}{|I_0\wedge \eta_n|}\sum_{k\in I_0\wedge \eta_n}\chi_{A_n}(x)\]

There are two cases:

\textbf{Case 1: }$\limsup_nF_n(x)\ge a$ for some $x$.  Put $I=\{n\in I_0:x\in A_n\}$, then $\{A_n\}_{n\in I}$ has the finite intersection property.  Observe:
\[\frac{|I\wedge \eta_n|}{\eta_n}=\left(\frac{|I\wedge \eta_n|}{|I_0\wedge \eta_n|}\right)\left(\frac{|I_0\wedge \eta_n|}{\eta_n}\right)>F_n(x) b\]
so if we let $n\to\infty$ along a subsequence $n_k$ witnessing $\limsup_nF_n(x)\ge a$, we get $\bar{d}(I)\ge ab$.

\textbf{Case 2: }$\limsup_nF_n(x)<a$ for all $x$.  Then for some $r<a$ and $C\in\SA$ with $\mu(C)>0$, $\limsup_nF_n(x)< r$ on $C$.  Let $\phi=r\chi_C+a\chi_{\compl{C}}$, and note that \[\limsup_nF_n(x)<\phi(x)\textrm{ for all }x\in X\label{limsup}.\]  Put:
\[E_0=\bigcup_{n\in\N} \big[\big( \mathop{\bigcap_{k\ge n}}_{k\in\N} {^*\{x\in X : F_k(x)<\phi(x)\}} \big)\cap{^*\big(\mathop{\bigcap_{k\ge n}}_{k\in\N} \{x\in X : F_k(x)<\phi(x)\}\big)}\big]\]
Observe that $E_0\in\SAS$.  If $x\in X$ is standard then by \ref{limsup}, $x\in E_0$.  It follows that $\mu_L(E_0)=1$, and we may take $B\in\SA$ with ${^*B}\subseteq E_0$ and $\mu(B)$ arbitrarily close to $1$.

For $x\in{^*B}$ let $n(x)$ be least so that $F_k(x)<{^*\phi}(x)$ for all $k\ge n(x), k\in{^*\N},$ and note that by definition of $E_0$ $n$ is an internal function taking finite values on ${^*B}$, so has a bound $N\in\N$ on ${^*B}$.  It follows:
\begin{align*}a&\le\frac{1}{|I_0\wedge \eta_N|}\sum_{k\in I_0\wedge \eta_N}\int\chi_{A_k}(x)=\int_XF_Nd\mu\\
&=\int_{B\cap C}F_Nd\mu+\int_{B\setminus C}F_Nd\mu+\int_{X\setminus B}F_Nd\mu\\
&\le r\mu(B\cap C) + a\mu(B\setminus C) + 1\mu(X\setminus B)\end{align*}

Letting $\mu(B)\to1$, $a\le r\mu(C)+a\mu(X\setminus C)=a-(a-r)\mu(C)<a$, a contradiction.  This completes the proof.

\section{Proof of Theorem~\ref{thm4}}
We are now ready to prove the main result.

(ii $\Rightarrow$ i) Let $\{x_k : k\in\mathbb{N}\}$, $I\subseteq\N$, and an $r>0$ as in (ii).

For all standard $n\in I$ and any {\em infinite} $k\in({^*\mathbb{N}}\setminus\mathbb{N}), |{^*f}_{n}(x_k)|>r$.  Fix such a $k$, and define $T\co B(X)\to\mathbb{R}$ by
$T(g)={\st{\,^*g}(x_k)}$.  It is easy to see that $T$ is a positive linear functional.  However, for standard $n\in\mathbb{N}$,
\[0<r<|{^*f}_{n}(x_k)|\approx |\st{\,^*f}_{n}(x_k)|=|T(f_{n})|\] so this $T$ and the same $I$ and $r$ from (ii) witness (i).

(i $\Rightarrow$ ii) Suppose (i) holds. Given the $T$, $I$, and $r>0$ given by (ii),
define a finite, finitely-additive measure on $(X, \mathcal{P}(X))$ by $\mu(E)=T(\chi_E)$.  Let $\bar{d}(I)>\alpha>0$.

Let $s<r$ and $\delta\in\mathbb{R}$ satisfy $0<\delta<{s}/{T(1)}$; equivalently, $0<T(\delta)<s$.  Note that for any $g\in B(X)$
with $-\delta\le g\le\delta$, positivity of $T$ ensures that
\[-T(\delta)=T(-\delta)\le T(g)\le T(\delta)\]
so $|T(g)|\le T(\delta)<s$.  Let $M>0$ be a bound for all the functions $f_n$.

For $n\in I$ put $A_{n}=\{x\in X : |f_{n}(x)|>\delta\}$.  Then
\[r<|T(f_{n})|=
|T(f_{n}\chi_{A_{n}})+T(f_{n}\chi_{\compl{A}_{n}})|\le |T(f_{n}\chi_{A_{n}})|+T(\delta)\le
MT(\chi_{A_{n}})+s\] so $\mu(A_{n})=T(\chi_{A_{n}})>\frac{r-s}{M}>0$ for all $n\in I$.  Note that by taking $s$ close to $0$ we can make this last term as close to $r/M$ as we like, in particular so that $\frac{r-s}{M}d(\bar{I})>\alpha\frac{r}{M}$.

By Corollary~\ref{bergelson}, there is a subset $J=\{n_m\}_m\subseteq I$ such that $\bar{d}(J)>\alpha^2{r/M}$
and such that for every $N\in\mathbb{N}$, $\mu\big(\bigcap\limits_{m=1}^NA_{n_m}\big)>0$.  Let $x_N\in\bigcap\limits_{m=1}^NA_{n_m}$.
For any $m, N\in\mathbb{N}$ with $N>m, x_N\in A_{n_m}$, therefore $|f_{n_m}(x_N)|>\delta$, so
for every $n\in J$, $\liminf\limits_{k\to\infty}|f_{n}(x_k)|\ge\delta$. The set $J$ and constant $\delta>0$ witness the implication (ii).  This completes the proof.

\bibliographystyle{plain}
\bibliography{dba4a}

\end{document}